\documentclass[letterpaper, reqno,11pt]{article}
\usepackage[margin=1.0in]{geometry}
\usepackage{color,latexsym,amsmath,amssymb, amsthm, graphicx,subfigure,revsymb, enumerate}
\usepackage[percent]{overpic}

\numberwithin{equation}{section}

\newcommand{\RR}{\mathbb{R}}
\newcommand{\CC}{\mathbb{C}}

\newcommand{\FF}{\mathbb{F}}

\newcommand{\ind}{\operatorname{Index}}
\newcommand{\forbid}{\operatorname{Forbid}}

\newtheorem{thm}{Theorem}[section]
\newtheorem{lem}[thm]{Lemma}

\newtheorem{cor}[thm]{Corollary}

\theoremstyle{remark}
\newtheorem{defn}[thm]{Definition}

\begin{document}
\pagenumbering{arabic}
\title{Improved Elekes-Szab\'o type estimates using proximity}
\author{Jozsef Solymosi\thanks{Department of Mathematics, The University of British Columbia. Vancouver, BC, Canada.} \and Joshua Zahl\footnotemark[1]}

\maketitle

\begin{abstract}
We prove a new Elekes-Szab\'o type estimate on the size of the intersection of a Cartesian product $A\times B\times C$ with an algebraic surface $\{f=0\}$ over the reals. In particular, if $A,B,C$ are sets of $N$ real numbers and $f$ is a trivariate polynomial, then either $f$ has a special form that encodes additive group structure (for example $f(x,y,x) = x + y - z$), or $A \times B\times C \cap\{f=0\}$ has cardinality $O(N^{12/7})$. This is an improvement over the previously bound $O(N^{11/6})$. We also prove an asymmetric version of our main result, which yields an Elekes-Ronyai type expanding polynomial estimate with exponent $3/2$. This has applications to questions in combinatorial geometry related to the Erd\H{o}s distinct distances problem.

Like previous approaches to the problem, we rephrase the question as a $L^2$ estimate, which can be analyzed by counting additive quadruples. The latter problem can be recast as an incidence problem involving points and curves in the plane. The new idea in our proof is that we use the order structure of the reals to restrict attention to a smaller collection of proximate additive quadruples.
\end{abstract}

\section{Introduction}
The Schwartz-Zippel lemma controls the size of the intersection of a Cartesian product and the zero-locus of a polynomial:

\begin{thm}[Schwartz-Zippel]
Let $F$ be a field, let $A_1,\ldots,A_k$ be subset of $F$ of size $N$, and let $f$ be a non-zero $k$-variate polynomial with coefficients in $F$. Then
\begin{equation}\label{SZBound}
| (A_1\times\ldots\times A_k) \cap Z(f) |\leq (\deg f)N^{k-1}. 
\end{equation}
\end{thm}
The bound \eqref{SZBound} can be tight, for example if $Z(f)$ is a union of parallel, axis-parallel hyperplanes. Motivated by questions in combinatorial geometry, Elekes and Szab\'o \cite{ES} investigated situations where Inequality \eqref{SZBound} can be strengthened. They were interested in the situation where $k$ and $\deg f$ are fixed, and $N$ is large.
\begin{defn}
Let $F$ be an infinite field. We say a $k$-variate polynomial $f$ with coefficients in $F$ has \emph{Schwartz-Zippel power saving} if there are constants $C,c>0$ so that for all $N\geq 1$ and all subsets $A_1,\ldots,A_k$ of $F$ of size $N$, we have
\begin{equation}
| (A_1\times\ldots\times A_k) \cap Z(f) |\leq C N^{d-c},\quad d = \dim Z(f). 
\end{equation}
\end{defn}

If $Z(f)$ is reducible, then $f$ has Schwartz-Zippel power saving if and only if all of the maximal-dimension irreducible components of $Z(f)$ have Schwartz-Zippel power saving. Thus it makes sense to consider the case where $f$ is irreducible. When $k=2$, no polynomials have Schwartz-Zippel power saving. When $k=3$, however, the situation is quite different. The following result of Elekes and Szab\'o \cite[Theorem 3]{ES} shows that an irreducible polynomial $f\in \CC[x,y,z]$ either must have Schwartz-Zippel power saving, or it must have a special structure.

\begin{thm}[Elekes-Szab\'o]\label{ESTheorem}
Let $f\in \CC[x,y,z]$ be irreducible. Then at least one of the following is true.
\begin{itemize}
\item[(A)] $f$ has Schwartz-Zippel power saving.
\item[(B)] After possibly permuting the coordinates $x,y,z$, we have $f(x,y,z) = g(x,y)$, for some bivariate polynomial $g$. 
\item[(C)] $f$ encodes additive group structure.
\end{itemize}
\end{thm}
If either of Items (B) or (C) hold, then Item (A) does not; in fact, for every $N$ there exist sets $A,B,C\subset\CC$ of size $N$ with $|(A\times B\times C)\cap Z(f)|\geq (N-2)^2/8$. Item (A) has already been defined, and item (B) is self-explanatory; geometrically, it says that $Z(f)$ is an axis-parallel cylinder above a curve. 

Item (C) requires additional explanation. An example of a polynomial that encodes additive group structure is $f(x,y,z) = x+y-z$. More generally, we say $f$ encodes additive group structure if for a generic point $p\in\CC^3$, there is a (Euclidean) neighborhood $U$ of $p$, a set $V\subset\CC$, and analytic functions $\phi\colon U\to\CC^3$ and $\psi\colon V\to\CC$, so that $\psi\circ f \circ \phi(x,y,z)=x+y-z$. When $f(x,y,z)$ is of the special form $h(x,y)-z$, then the situation is particularly simple: $f$ encodes additive structure of and only if $h$ has the form $h(x,y) = a(b(x)+c(y))$ or $h(x,y) = a(b(x)c(y))$ for univariate polynomials $a,b,c$. This special case was analyzed in an earlier work by Elekes and R\'onyai \cite{ER}. We refer the reader to \cite{ES} for further discussion.  

Theorem \ref{ESTheorem} gives a satisfactory qualitative description of the Schwartz-Zippel power savings phenomena in $\CC^3$. In \cite{BB}, Bays and Breuillard generalized Theorem \ref{ESTheorem} by characterizing which varieties in $\CC^k$ have a Schwartz-Zippel power saving. In \cite{Tao} Tao considered a related problem in $\FF_p^3$ for the Elekes-R\'onyai type situation $f(x,y,z) = h(x,y)-z$. 

We will restrict attention to three dimensions, and characteristic 0. We will be interested in quantitative versions of Theorem \ref{ESTheorem}, and specifically, we wish to obtain explicit lower bounds on the size of the Schwartz-Zippel power saving. In this direction, Raz, Sharir, and de Zeeuw \cite{RSZ} strengthened Theorem \ref{ESTheorem} by establishing the explicit power saving $c=1/6$ for Item (A). The proof in \cite{RSZ} generalized several related arguments that had been previously used to obtain the same power savings in certain special cases \cite{RSS2, RSS, SSS, SS}. In the other direction, Makhul, Roche-Newton, Warren, and de Zeeuw \cite{MRWZ} obtained an upper bound on the size of the Schwartz-Zippel power saving by showing that the polynomial $f(x,y,z)=(x-y)^2+x-z$ only has Schwartz-Zippel power saving $c=1/2$. Our main result is a version of Theorem \ref{ESTheorem} in three dimensions for Cartesian products of real numbers, with power saving $c=2/7$. In what follows, we identify points $x\in\RR$ with the corresponding point $x+0i\in\CC$. 

\begin{thm}\label{mainThm}
Let $f\in \CC[x,y,z]$ be irreducible. Then at least one of the following is true.
\begin{itemize}

	\item[(A)] For all finite sets $A,B,C\subset\RR$ with $|A|\leq |B|\leq |C|$, we have
	\begin{equation}\label{bdABC}
	| (A\times B \times C)\cap Z(f)|\lesssim (|A| |B| |C|)^{4/7} + |B| |C|^{1/2},
	\end{equation}
	where the implicit constant depends on the degree of $f$.

	\item[(B)] After possibly permuting the coordinates $x,y,z$, we have $f(x,y,z) = g(x,y)$, for some bivariate polynomial $g$. 

	\item[(C)] $f$ encodes additive group structure.
\end{itemize}
\end{thm}
Specializing to the case $f(x,y,z)=h(x,y)-z$, we record the following corollary.
\begin{cor}
Let $h\in \CC[x,y]$. Then exactly one of the following holds.
\begin{itemize}
	\item[(A)] For all finite sets $A,B\subset\RR$ with $|A|\leq|B|$, we have
	\begin{equation}\label{growthAB}
		|h(A\times B)|\gtrsim \min\big( |A|^{3/4}|B|^{3/4},\ |A|^2\big). 
	\end{equation}

	\item[(B)] $h$ has the special form $h(x,y) = a(b(x)+c(y))$ or $h(x,y) = a(b(x)c(y))$ for univatiate polynomials $a,b,c$. 
\end{itemize}
\end{cor}

As mentioned above, Theorem \ref{ESTheorem} was motivated by problems in combinatorial geometry. Our quantitative strengthening \eqref{bdABC} leads to improved estimates for some of these problems. We will discuss this below.

\subsection{Applications to combinatorial geometry}
\paragraph{Distinct distances on two lines.} Let $L_1,L_2$ be two lines in the plane that are not parallel or orthogonal. Let $P_1\subset L_1$ and $P_2\subset L_2$, with $|P_1|\leq |P_2|$. Then there are $\gtrsim \min(|P_1|^{3/4}|P_2|^{3/4},\ |P_1|^2)$ distinct distances between the points in $P_1$ and the points in $P_2$. This problem was originally proposed by Purdy \cite[Section 5.5]{BWP}, and was previously studied by Elekes \cite{El} who obtained the bound $\Omega(n^{5/4})$ in the special case $|P_1|=|P_2|=n$, and by Sharir, Sheffer, and the third author \cite{SSS} who obtained the bound $\gtrsim \min(|P_1|^{2/3}|P_2|^{2/3},\ |P_1|^2)$.

\paragraph{Distinct distances from three points.} Let $p_1,p_2,p_3\in\RR^2$ be three points that are not collinear. Let $P\subset\RR^2$ be a set of $n$ points. Then there are $\Omega(n^{7/12})$ distinct distances between the points in $\{p_1,p_2,p_3\}$ and the points in $P$. This problem was previously considered by Elekes and Szab\'o \cite{ES} who obtained the bound $\Omega(n^{0.502})$, and by Sharir and the second author \cite{SS} who obtained the bound $\Omega(n^{6/11}).$

\paragraph{Distinct distances on an algebraic curve.} Let $\gamma\subset\RR^2$ be an algebraic curve of degree $d$ that does not contain a line or circle. Let $P\subset\gamma$ have size $n$. Then $P$ determines $\Omega(n^{3/2})$ distinct distances. This problem was previously considered by Charalambides \cite{Ch} who obtained the bound $\Omega(n^{5/4})$ and by Pach and de Zeeuw \cite{PZ} who obtained the bound $\Omega(n^{4/3}).$

\paragraph{Triple points for unit circles.} Let $p_1,p_2,p_3\in\RR^2$ be three distinct points. For $i=1,2,3$ let $\cal{C}_i$ be a set of $n$ unit circles that contain the point $p_i$. Then there are $O(n^{12/7})$ points in $\RR^2$ that are incident to a circle from each of $\mathcal{C}_1$, $\mathcal{C}_2$, and $\mathcal{C}_3$. An estimate of the form $o(n^2)$ was conjectured by Székely \cite[Conjecture 3.41]{El2}. A generalization of this conjecture was  proved by Elekes, Simonovits, and Szab\'o \cite{ESS}, who obtained the bound $O(n^{2-\eta})$ for a small positive $\eta>0$ for more general families of curves. The problem was also considered by Raz, Sharir and the first author \cite{RSS2} who obtained the bound $\Omega(n^{11/6})$.


\section{Counting quadruples with proximity}
\begin{defn}
Let $A=\{a_1,\ldots,a_n\}\subset\RR$, where the elements $a_i$ are written in increasing order. For $a\in A$, we define $\ind_A(a)$ to be the unique index $i$ so that $a = a_i$. Sometimes we will write $\ind(a)$ if the ambient set $A$ is apparent from the context. 
\end{defn}

\begin{lem}\label{smallGapMontone}
Let $S\geq 1$ and let $A,B\subset\RR$ be finite sets. For each $a\in A$, let $\forbid(a)\subset A$ be a set of size $<S$. Similarly, for each $b\in B$, let $\forbid(b)\subset B$ be a set of size $<S$. Let $I\subset\RR$ be an interval, let $f\colon I\to\RR$ be (weakly) monotone, and let $G\subset (A\times B)\cap \{f=0\}$. Then there exist at least $|G|/(2S)-1$ quadruples $(a,a',b,b')\in A^2\times B^2$ that satisfy
\begin{equation}\label{goodQuadrupleMonotoneF}
\begin{split}
&(a,b),(a',b')\in G,\quad a'\not\in\forbid(a),\ b'\not\in\forbid(b),\\
&|\ind(a)-\ind(a')|\leq \frac{4S|A|}{|G|},\quad |\ind(b)-\ind(b')|\leq \frac{4S|B|}{|G|}.
\end{split}
\end{equation}
\end{lem}
\begin{proof}
Index the elements of $A$ and $B$ in increasing order. Write the points of $G$ in increasing order of $x$-coordinate as $(a_{n(i)}, b_{m(i)})$, $i=0,\ldots,|G|-1$, where $n(i)$ is strictly increasing and $m(i)$ is monotone. Let $N = \lfloor |G|/S\rfloor-1$. For each $0\leq j\leq N$, we say the index $i=Sj$ has a big $x$-skip if $n(i + S)-n(i)>4S|A|/|G|$. Similarly, we say $i$ has a big $y$-skip if $|m(i+S)-m(i)|>4S|B|/|G|$. There are at most $|G|/(4S)$ indices $0\leq j\leq N$ with a big $x$-skip, and at most $|G|/(4S)$ indices with a big $y$-skip. Thus there are at least $|G|/(2S)-1$ indices that have neither a big $x$-skip nor a big $y$-skip. For each such index $i=Sj$, there must exist indices $i<k,k'\leq i+S$ so that $a_{n(k)}\not\in \forbid(a_{n(i)})$, and $b_{m(k')}\not\in \forbid(b_{m(i)})$. Thus the quadruple $\big(a_{n(i)}, a_{n(k)}, b_{m(i)}, b_{m(k')}\big)$ satisfies \eqref{goodQuadrupleMonotoneF}. Different indices $i = Sj$ give rise to different quadruples, since $n(i)$ is strictly monotone. 
\end{proof}

By B\'ezout's theorem and Harnack's curve theorem, if $f\in\CC[x,y]$ then the real part of its zero-locus $\mathcal{R}(Z(f))$ can be written as the union of $\leq 2(\deg f)^2$ curves, each of which is the graph of a monotone function $y=g(x)$ or $x=g(y)$ above an interval, plus a union of at most $2(\deg f)^2$ points. Thus Lemma \ref{smallGapMontone} has the following corollary. In what follows, all implicit constants are allowed to depend on $\deg f$.
\begin{cor}\label{smallGapPoly}
Let $S\geq 1$ and let $A,B\subset\RR$ be finite sets. For each $a\in A$, let $\forbid(a)\subset A$ be a set of size $<S$. Similarly, for each $b\in B$, let $\forbid(b)\subset B$ be a set of size $<S$. Let $f\in\CC[x,y]$ and let $G\subset (A\times B)\cap Z(f)$. Then there exist $\gtrsim|G|/S-1$ quadruples $(a,a',b,b')\in A^2\times B^2$ that satisfy
\begin{equation}\label{goodQuadruplePolynomialF}
\begin{split}
&(a,b),(a',b')\in G,\quad a'\not\in\forbid(a),\ b'\not\in\forbid(b),\\
&|\ind(a)-\ind(a')|\lesssim \frac{S|A|}{|G|},\quad |\ind(b)-\ind(b')|\lesssim \frac{S|B|}{|G|}.
\end{split}
\end{equation}
\end{cor}

\begin{lem}\label{proximity5Tuples}
Let $S\geq 1$ and let $A,B,C\subset\RR$ be finite sets. For each $a\in A$, let $\forbid(a)\subset A$ be a set of size $<S$. Similarly, for each $b\in B$, let $\forbid(b)\subset B$ be a set of size $<S$. Let $f\in\CC[x,y,z]$ and let $G\subset (A\times B\times C)\cap Z(f)$. Then there exist $\gtrsim |G|/S-|C|$ tuples $(a,a',b,b',c)\in A^2\times B^2\times C$ that satisfy
\begin{equation}
\begin{split}
&(a,b,c),\ (a',b',c)\ \in G,\quad a'\not\in\forbid(a),\ b'\not\in\forbid(b),\\
&|\ind(a)-\ind(a')|\lesssim \frac{S|A||C|}{|G|},\quad|\ind(b)-\ind(b')|\lesssim \frac{S|B||C|}{|G|}.
\end{split}
\end{equation}
\end{lem}
\begin{proof}
Let 
\[
C'=\Big\{c\in C\colon |(A\times B\times\{c\}) \cap G|\geq \frac{|G|}{2|C|}\Big\}.
\]
Then $|(A\times B\times C')\cap G|\geq|G|/2$.  For each $c\in C'$, apply Corollary \ref{smallGapPoly} to $A,B$, and $g(x,y)=f(x,y,c)$. 
\end{proof}


\section{From point-curve incidences to Schwartz-Zippel bounds}
In this section, we will fix a non-zero polynomial $f\in \RR[x,y,z]$ that does not satisfy Conclusions (B) or (C) from Theorem \ref{mainThm}. In what follows, all implicit constants are allowed to depend on the degree of $f$. 

For $a,a'\in \CC$, we define
\begin{equation}
\begin{split}
\alpha_{a,a'} & = \{(y,y',z)\in\CC^3\colon f(a,y,z) = 0,\ f(a',y',z)=0\},\\
\beta_{a,a'} &= \{(y,y')\in\CC^2\colon \exists\ z\in\CC\ \operatorname{s.t.}\ f(a,y,z)=f(a',y',z)=0 \}.
\end{split}
\end{equation}
$\alpha_{a,a'}$ is a (affine) variety, and $\beta_{a,a'}$ is the projection of $\alpha$ to the first two coordinates; hence it is a constructible set. We will denote its Zariski closure by $\gamma_{a,a'}=\overline{\beta}_{a,a'}$. Similarly, for $b,b'\in \CC$, we define $\gamma^*_{b,b'}$ to be the Zariski closure of the set
\[
\{(x,x')\in\CC^2\colon \exists\ z\in\CC\ \operatorname{s.t.}\ f(x,b,z)=f(x',b',z)=0 \}.
\]

We next recall Lemma 2.3 from \cite{RSZ}. In what follows, we make crucial use of the assumption that $f$ does not satisfy Conclusions (B) or (C) from Theorem \ref{mainThm}. 

\begin{lem}
For all pairs $(a,a')\in\CC^2$ outside a ``bad'' set of size $O(1)$, the variety $\gamma_{a,a'}$ is either empty or is a curve of degree $O(1)$. An analogous result holds for the curves $\gamma^*_{b,b'}$.
\end{lem}

Following the arguments in \cite{RSZ}, we will use results from incidence geometry to control incidences between points $p\in B\times B$ and certain curves $\alpha_{a,a'}$ with $a,a'\in A$. One difficulty is that distinct pairs $(a,a')$ might lead to curves $\gamma_{a,a'}$ that contain a common irreducible component. Our next task is to show that this situation cannot occur frequently. We recall Definition 2.4 and Proposition 2.5 from \cite{RSZ}.

\begin{defn}
We say a curve $\gamma\subset\CC^2$ is \emph{popular} if there are at least $(\deg f)^4+1$ points $(a,a')\in\CC^2$ so that $\gamma\subset \gamma_{a,a'}$. We say a pair $(a,a')\in\CC^2$ is \emph{dangerous} if $\gamma_{a,a'}$ contains a popular curve. Otherwise, it is \emph{safe}. We define dangerous and safe pairs $(b,b')\in\CC^2$ analogously. 
\end{defn}

\begin{lem}
The set of dangerous points in $A\times A$ is contained in a curve $\zeta\subset\CC^2$ of degree $O(1)$. Similarly, the set of dangerous points in $B\times B$ is contained in a curve $\zeta^*\subset\CC^2$ of degree $O(1)$.
\end{lem}

\begin{cor}\label{smallForbiddenSet}
For each $a\in A$ there is a set $\operatorname{forbid}(a)\subset A$ of size $O(1)$, so that if $(a,a')\in A^2$ with $a'\not\in\forbid(a)$, then $(a,a')$ is safe. An analogous result holds for pairs $(b,b')\in B^2$. 
\end{cor}

Combining Lemma \ref{proximity5Tuples} and Corollary \ref{smallForbiddenSet}, we obtain the following.
\begin{lem}\label{manySafe5Tupples}
There is a constant $K = O(1)$ so that the following holds. Let $G = (A\times B\times C)\cap Z(f)$. Then there are $\gtrsim |G|$ tuples $(a,a',b,b',c)$ with the following properties.
\begin{itemize}
	\item $(a,b,c),\ (a',b',c)\in G$.
	\item $(a,a')$ and $(b,b')$ are safe.
	\item $|\ind(a)-\ind(a')|\leq K |A||C|/|G|$.
	\item $|\ind(b)-\ind(b')|\leq K |B||C|/|G|$.
\end{itemize}
\end{lem}

Next, with $K$ as above we define
\begin{equation}
\begin{split}
\Gamma & = \{\gamma_{a,a'}\colon a\in A,\ a'\in A\backslash \forbid(a),\ |\ind(a)-\ind(a')|\leq K|A||C|/|G|\},\\
P & = \{(b,b') \colon b\in B,\ b'\in B\backslash \forbid(b),\ |\ind(b)-\ind(b')|\leq K|B||C|/|G|\}.
\end{split}
\end{equation}
$\Gamma$ is a collection of bounded degree curves in $\CC^2$, and $P$ is a subset of $\RR^2$, which we identify with a subset of $\CC^2$. We have 
\begin{equation}\label{sizeGammaAndP}
|\Gamma|\lesssim |A|^2|C|/|G|,\quad |P|\lesssim |B|^2|C|/|G|.
\end{equation}
For each pair of curves $\gamma,\gamma'\in\Gamma$, we have $|\gamma\cap\gamma'|\lesssim 1$ (this precisely the statement that each curve $\gamma_{a,a'}\in\Gamma$ is constructed from a safe pair $(a,a')$ ). Similarly, for each pair of points $p,p'\in P$, there are $\lesssim 1$ curves in $\Gamma$ that contain both $p$ and $p'$ (this precisely the statement that each pair $(b,b')\in P$ is safe). For such an arrangement of points and curves, we have the following Szemer\'edi-Trotter type bound:
\begin{equation}\label{incidenceBound}
I(P,\Gamma)\lesssim |P|^{2/3}|\Gamma|^{2/3}+|P|+|\Gamma| \lesssim \Big(\frac{|A||B||C|}{|G|}\Big)^{4/3} + \frac{|B|^2|C|}{|G|},
\end{equation}
where the final inequality used the assumption that $|A|\leq |B|$. Note that while $\Gamma$ is a set of complex plane curves, the coordinates of the points in $P$ are real, and thus we may restrict each plane curve $\gamma\in\Gamma$ to its real locus $\mathcal{R}(\gamma)$. Thus Inequality \eqref{incidenceBound} is an incidence bound for sets of points and algebraic curves in $\RR^2$. Bounds of this form go back to the pioneering work of CEGSW \cite{CEGSW}; see also \cite{PS}.

On the other hand, each tuple from Lemma \ref{manySafe5Tupples} contributes an incidence to $I(P,\Gamma)$. We conclude that
\begin{equation}
|G|\lesssim I(P,\Gamma)\lesssim \Big(\frac{|A||B||C|}{|G|}\Big)^{4/3} + \frac{|B|^2|C|}{|G|}.
\end{equation}
rearranging, we obtain \eqref{bdABC}.

\section{Acknowledgements}
Research of JS was supported in part by a Hungarian National Research Grant KKP grant no. 133819, by an NSERC Discovery grant and OTKA K grant
no.119528. Research of JZ was supported by an NSERC Discovery grant.

\end{document}